\documentclass[11pt]{amsart}
\usepackage{verbatim}
\usepackage{amsmath}
\usepackage{enumerate}
\usepackage{amssymb}
\usepackage{color}
\usepackage{url}
\usepackage{graphicx}
\usepackage{fullpage}

\newcommand\+{\;\lower\plusheight\hbox{$+$}\;}
\newcommand\lldots{\;\lower\plusheight\hbox{$\cdots$}\;}

\newcommand{\Z}{\mathbb{Z}}
\newcommand{\Q}{\mathbb{Q}}

\newcommand{\RR}{\mathbb{R}}

\newcommand{\D}{\mathbb{D}}

\newtheorem{Theorem}{Theorem}[section]
\newtheorem{Lemma}[Theorem]{Lemma}
\newtheorem{Corollary}[Theorem]{Corollary}
\newtheorem{Definition}[Theorem]{Definition}
\newtheorem{Example}[Theorem]{Example}
\newtheorem{Remark}[Theorem]{Remark}
\newtheorem{Proposition}[Theorem]{Proposition}

\def\Re{\mathop{\rm Re}\nolimits}

\def\R{\mathop{\rm R}\nolimits}

\setbox0=\hbox{$+$}
\newdimen\plusheight
\plusheight=\ht 0 \setbox0=\hbox{$-$}
\newdimen\minusheight
\minusheight=\ht0

\setbox0=\hbox{$\cdots$}
\newdimen\cdotsheight
\cdotsheight=\plusheight

\begin{document}

\title{Gross-Zagier Type CM Value Formulas on $X_{0}^{*}(p)$}
\author{Dongxi Ye}
\dedicatory{To my father Yuanchang Ye on the occasion of his 55th birthday}
\address{
Department of Mathematics, University of Wisconsin\\
480 Lincoln Drive, Madison, Wisconsin, 53706 USA.}
\email{lawrencefrommath@gmail.com}
\maketitle
\begin{abstract}
In this work, we derive Gross-Zagier type CM value formulas for Hauptmoduls $j_{p}^{*}(\tau)$ on Fricke groups $\Gamma_{0}^{*}(p)$. We also illustrate how to employ these formulas to obtain certain Hilbert class polynomials.
\end{abstract}
\noindent
\numberwithin{equation}{section}
\allowdisplaybreaks
\section{Introduction}
\label{intro}

Denote by $\mathbb{H}$ the upper half plane, i.e., the set of complex numbers with positive imaginary part, and let $j(\tau)$ be the famous modular $j$-invariant defined on $\mathbb{H}$. As a modular function, it can be seen that $j(\tau)$ generates the function field of the compactified modular curve $X(1)$ on $\rm{SL}_{2}(\mathbb{Z})$, the multiplicative group of $2\times2$ matrices over $\mathbb{Z}$ of determinant one. In the algebraic respects, the value of $j(\tau)$ at an imaginary quadratic point $\tau_{0}$ (usually called a CM point) generates some ring class field of the imaginary quadratic field $\mathbb{Q}(\tau_{0})$ (see, e.g., \cite{B, C, L, Sh}). For example, let $-d$ be a negative fundamental discriminant, and let $\tau_{Q}$ be the unique imaginary quadratic point in $\mathbb{H}$ arising from a quadratic form $Q(X,Y)=aX^{2}+bXY+cY^{2}$ of discriminant $-d$, then the value $j(\tau_{Q})$ generates the Hilbert class field (the maximal unramified abelian extension) of $\mathbb{Q}(\tau_{Q})$. More interestingly, the Galois conjugates of $j(\tau_{Q})$ are simply $j(\tau_{Q'})$, where $Q'$ runs over the equivalence classes of the set $\mathcal{Q}_{d}^{prim,+}$ of primitive and positive-definite quadratic forms of discriminant $-d$ modulo the action of ${\rm SL}_{2}(\mathbb{Z})$. Such values are algebraic integers of degree $h(-d)$, the class number of $\mathbb{Q}(\sqrt{-d})$, and are classically called  singular moduli (see, e.g., \cite{W, Z}). 

In their seminal work \cite{GZ}, Gross and Zagier consider the product
\begin{equation}
\label{gzprod}
\prod_{Q_{D}\in\mathcal{Q}_{D}^{prim,+}/\rm{SL}_{2}(\mathbb{Z})}\prod_{Q_{d}\in\mathcal{Q}_{d}^{prim,+}/\rm{SL}_{2}(\mathbb{Z})}\left(j(\tau_{Q_{D}})-j(\tau_{Q_{d}})\right)
\end{equation}
where $-D$ is a negative fundamental discriminant coprime to $-d$, which is the rational norm (up to sign) of the  difference of singular moduli $j(\tau_{Q_{D}})-j(\tau_{Q_{d}})$, so a rational integer, and they prove a remarkable formula that equivalently gives the factorization of this rational integer, namely,
$$
\prod_{Q_{D}\in\mathcal{Q}_{D}^{prim,+}/\rm{SL}_{2}(\mathbb{Z})}\prod_{Q_{d}\in\mathcal{Q}_{d}^{prim,+}/\rm{SL}_{2}(\mathbb{Z})}|j(\tau_{Q_{D}})-j(\tau_{Q_{d}})|^{2}=\prod_{\substack{x,n,n'\in\mathbb{Z}\\n,n'>0\\x^{2}+4nn'=dD}}n^{\epsilon(n')},
$$
where $\epsilon(n')$ is multiplicative, and $\epsilon(p)$ is defined via the local Hilbert symbol at a prime $p$, and here we have assumed $-d,-D<-4$ for simplicity. This formula beautifully reveals the arithmetic information encoded in the prime factor exponents of the rational integer \eqref{gzprod} via local Hilbert symbols, and we now call it Gross-Zagier CM value formula. 

Gross and Zagier give an analytic and an algebraic proof to their CM value formula, and since then, it has been reproved by various mathematicians by different methods (see, e.g., \cite{D, LV, S, YY}), of which an interesting one is from the point of view of Borcherds forms \cite{BO,K}. In his thesis \cite{S2}, Schofer uses regularized theta lift and Siegel-Weil formula to compute and express the trace of a Borcherds form on an orthogonal Shimura variety of type $\rm{O}(n,2)$ over a so-called small CM-cycle (the term ``small'' is used to distinguish it from the big CM-cycle \cite{BKY}) in terms of the Fourier coefficients of an incoherent Eisenstein series of weight one. Now we call the formula representing Schofer's result small CM value formula. This can be viewed as a generalization of Gross-Zagier formula in  view of that
\begin{equation}
\label{gzbor}
\prod_{Q_{d}\in\mathcal{Q}_{d}/\rm{SL}_{2}(\mathbb{Z})}\left(j(\tau)-j(\tau_{Q_{d}})\right)=\Psi(z;\vec{F})
\end{equation}
is a Borcherds form for $\rm{O}(1,2)$, and the rational norm \eqref{gzprod} turns out to be
$$
\prod_{Q_{D}\in\mathcal{Q}_{D}^{prim,+}/\rm{SL}_{2}(\mathbb{Z})}\prod_{Q_{d}\in\mathcal{Q}_{d}^{prim,+}/\rm{SL}_{2}(\mathbb{Z})}|j(\tau_{Q_{D}})-j(\tau_{Q_{d}})|=\prod_{z\in Z(U)_{K}}|\Psi(z;\vec{F})|
$$
where the product on the right is over a CM-cycle $Z(U)_{K}$ of the underlying Shimura variety for $\rm{O}(1,2)$ that can be identified with $\mathcal{Q}_{D}^{prim,+}/\rm{SL}_{2}(\mathbb{Z})$ on the left by some means. Then applying Schofer's small CM value formula to the product on the right and computing relevant Fourier coefficients of an incoherent Eisenstein series of weight one, one can obtain an equivalent version of Gross-Zagier CM value formula.

In this work, we aim to apply Schofer's small CM value formula to the $\Gamma_{0}^{*}(p)$ analogues of \eqref{gzbor} to extend Gross-Zagier formula to Hauptmoduls on $\Gamma_{0}^{*}(p)$, where $\Gamma_{0}^{*}(p)$ is the group, called Fricke group, generated by the Hecke group $\Gamma_{0}(p)$, i.e.,
$$
\Gamma_{0}(p)=\left\{\left.\begin{pmatrix}a&b\\c&d\end{pmatrix}\in\rm{SL}_{2}(\mathbb{Z})\right|\,c\equiv0\pmod{p}\right\},
$$
and the Fricke involution $\begin{pmatrix}0&-\frac{1}{\sqrt{p}}\\ \sqrt{p}&0\end{pmatrix}$. Recall that for a genus zero congruence subgroup $\Gamma$ of ${\rm SL}_{2}(\RR)$ commensurable with ${\rm SL}_{2}(\Z)$, the function field on $X(\Gamma)$ can be generated by a single modular function, and such function is called a Hauptmodul for $\Gamma$ if it has a unique simple pole of residue~1 at the cusp $i\infty$, i.e., it has Fourier expansion of the form $q^{-1/h}+c(0)+c(1)q^{1/h}+\cdots$ with $q=\exp(2\pi i\tau)$ at the cusp $i\infty$ where $h$ is the width of the cusp $i\infty$. It is known by the theory of complex multiplication (see, e.g., \cite{Sh}) that similar to that of $j(\tau)$, the value of $j_{p}^{*}(\tau)$ at some special imaginary quadratic point, called Heegner point, generates the Hilbert class field of the associated imaginary quadratic field, so it will be interesting to derive a Gross-Zagier type formula for $j_{p}^{*}(\tau)$ over Heegner points. Now we come to the main result of this work. 

\begin{Theorem}
\label{main}
Let $p$ be a prime such that the Fricke group $\Gamma^{*}_{0}(p)$ is of genus zero, and let $j_{p}^{*}(\tau)$ be a Hauptmodul on $\Gamma^{*}_{0}(p)$. Let $-d,-D<-4$ be two distinct negative fundamental discriminants (not necessarily coprime). Let $\mu,\beta\in\mathbb{Z}/2p\mathbb{Z}$, and  let $\mathcal{Q}_{D,p,\mu}^{prim,+}$ denote the set of primitive and positive-definite quadratic forms $aX^{2}+bXY+cY^{2}$ of discriminant $-D$ with $p|a$ and $b\equiv\mu\pmod{2p}$. The set $\mathcal{Q}_{d,p,\beta}^{prim,+}$ is defined similarly. For $y,n,B\in\mathbb{Z}$, write $g=\gcd(\mu,2p)$ and
$$
m(\beta,y,n)=\frac{d}{4p}-\frac{(g\mu\beta-2npD-2gpy)^{2}}{4g^{2}pD}.
$$
Then
\begin{align*}
&\log\left(\prod_{Q_{D}\in \mathcal{Q}^{prim,+}_{D,p,\mu}/\Gamma_{0}(p)}\prod_{Q_{d}\in \mathcal{Q}^{prim,+}_{d,p,\beta}/\Gamma_{0}(p)}|j_{p}^{*}(\tau_{Q_{D}})-j_{p}^{*}(\tau_{Q_{d}})|^{8}\right)\\
&=\sum_{\substack{n\in\mathbb{Z}\\0\leq y<D/g\\|g\mu\beta-2npD-2gpy|<g\sqrt{dD}}}2^{o(m(\beta,y,n))}\Bigg[\sum_{\substack {\text{$q$ inert in $k$}\\q\in{\rm Diff}\left(m(\beta,y,n)\right)}}\left({\rm ord}_{q}(m(\beta,y,n))+1\right)\rho(m(\beta,y,n)D/q)\log{q}\\
&\hspace{7cm}+\sum_{\substack{q|D\\q\in{\rm Diff}(m(\beta,y,n))}}{\rm ord}_{q}(m(\beta,y,n))\rho(m(\beta,y,n)D)\log{q}\Bigg]\\
&\quad+\sum_{\substack{n\in\mathbb{Z}\\0\leq y<D/g\\|g\mu(-\beta)-2npD-2gpy|<g\sqrt{dD}}}2^{o(m(-\beta,y,n))}\Bigg[\sum_{\substack {\text{$q$ inert in $k$}\\q\in{\rm Diff}\left(m(-\beta,y,n)\right)}}\left({\rm ord}_{q}(m(-\beta,y,n))+1\right)\rho(m(-\beta,y,n)D/q)\log{q}\\
&\hspace{7cm}+\sum_{\substack{q|D\\q\in{\rm Diff}(m(-\beta,y,n))}}{\rm ord}_{q}(m(-\beta,y,n))\rho(m(-\beta,y,n)D)\log{q}\Bigg],
\end{align*}
where $o(m)$, ${\rm Diff}(m)$ and $\rho(m)$ are defined in Theorem \ref{coeff}.

\end{Theorem}

\begin{Remark}
The condition $-d,-D<-4$ is only for simplicity, and in general is not necessary.
\end{Remark}

\begin{Remark}
Similar Gross-Zagier type formulas for Hauptmoduls $j_{p}(\tau)$ on $\Gamma_{0}(p)$ are derived in \cite{YY} by Yang and Yin for $p=2$ and in another recent work of the author \cite{Y} for $p\geq3$ using the so-called big CM value formula \cite{BKY}, while the Gross-Zagier type formulas obtained in \cite{YY,Y} are over some big CM-cycles which are sets of some CM points related to some ring class group of conductor $p$ instead of Heegner points. This distinction results from the difference between the computations of  the so-called big CM-cycle and the small CM-cycle \cite{BKY}.  The reason for why the small CM value formula may not work for $j_{p}(\tau)$ over a small CM-cycle is briefly explained in Remark \ref{rem2}.
\end{Remark}

As we have mentioned previously that for a Heegner point $\tau_{Q_{d}}$ of discriminant $-d$, the value $j_{p}^{*}(\tau_{Q_{d}})$ generates the Hilbert class field of $\mathbb{Q}(\sqrt{-d})$ and its Galois conjugates are $j_{p}^{*}(\tau_{Q_{d}'})$ as $Q_{d}'$ ranges over the equivalence classes of $\mathcal{Q}_{d,p,\beta}^{prim,+}$ modulo $\Gamma_{0}(p)$, so the minimal polynomial of $j_{p}^{*}(\tau_{Q_{d}})$ is simply
$$
Y=\prod_{Q_{d}\in\mathcal{Q}_{d,p,\beta}^{prim,+}/\Gamma_{0}(p)}\left(X-j_{p}^{*}(\tau_{Q_{d}})\right),
$$
and thus is a Hilbert class polynomial of $\mathbb{Q}(\sqrt{-d})$. Constructions of Hilbert class polynomials have been intensively studied (see, e.g., \cite{BBEL, Bo, En, Su}). At the end of this work, we will also illustrate how Theorem~\ref{main} is employed to compute and construct certain Hilbert class polynomials. One example we obtain is
\begin{equation}
\label{ex1}
Y=X^{4}-X^{3}+2X^{2}-2X+1,
\end{equation}
which is a Hilbert class polynomial of $\mathbb{Q}(\sqrt{-39})$, and whose root
$$
X=\frac{1}{4}-\frac{1}{4}\sqrt{-3}+\frac{1}{4}\sqrt{-10+6\sqrt{-3}}
$$
 generates the Hilbert class field of $\mathbb{Q}(\sqrt{-39})$.

\begin{Remark}
As pointed out by an anonymous reader, computing class polynomials using Hauptmoduls $j_{p}^{*}(\tau)$ for $\Gamma_{0}^{*}(p)$ is not new and has been intensively studied by Enge and Sutherland \cite{ES} via CRT method. In addition, in contrast to the limitation (see \eqref{limit}) in our way, Enge and Sutherland have shown how the CRT method can be used to compute class polynomials induced from $j_{p}^{*}(\tau)$ for negative discriminants with arbitrarily large class numbers. Later, more sophisticated methods were developed and used by Sutherland \cite{SU} to compute new examples of class polynomials for these class invariants on $X_{0}^{*}(p)$. The example \eqref{ex1} can be obtained by implementing the algorithms described in the work mentioned above.
\end{Remark}

This work is organized as follows. In Section \ref{adel}, we review the adelic formulation of the theory of Borcherds forms, and state the $\Gamma_{0}^{*}(p)$ extensions of \eqref{gzbor}. In Section \ref{schofer}, we state Schofer's small CM value formula and other related concepts such as CM-cycle and Fourier coefficients of an incoherent Eisenstein series. In Section 4, we show how to identify a CM-cycle in consideration with a set of imaginary quadratic points, and do some necessary lattice computations. Proof of Theorem~\ref{main} is given at the end of Subsection~4.1. In the last section, we explain how to apply Theorem~\ref{main} to construct certain Hilbert class polynomials.

\section{Adelic Formulation of Borcherds Forms}
\label{adel}
In this section, we review the theory of Borcherds forms in the adelic setting \cite{K} and state the $\Gamma_{0}^{*}(p)$ analogues of \eqref{gzbor}.

\subsection{Rational Quadratic Space}
Let $V$ be a vector space with quadratic form $Q$ of signature $(n,2)$. For a $\mathbb{Q}$-algebra $F$, we write $V(F)=V\otimes_{\mathbb{Q}}F$. Let $\mathbb{D}$ denote the Grassmannian of oriented negative 2-planes of $V(\mathbb{R})$. Then $\mathbb{D}$ is a symmetric space for $O(n,2)$ and has a Hermitian structure. It can be viewed as an open subset $\mathcal{L}$ of a quadric in $\mathbb{P}^{1}(V)(\mathbb{C})$. Explicitly,
$$
\mathbb{D}\cong\mathcal{L}:=\{w\in\mathbb{P}^{1}(V)(\mathbb{C})|\,(w,w)=0,\,(w,\bar{w})<0\}/\mathbb{C}^{\times}
$$
where the isomorphism is given by $[x,-y]\to x+iy$ for a properly oriented basis $[x,-y]$. Let $H={\rm GSpin}(V)$ be the general spin group of $V$. Let $\mathbb{A}$ be the adele ring over $\mathbb{Q}$ and $\mathbb{A}_{f}$ be the associated finite adele ring. Assume $K$ to be an open compact subgroup of $H(\mathbb{A}_{f})$ such that $H(\mathbb{A})=H(\mathbb{Q})H(\mathbb{R})^{+}K$, where $H(\mathbb{R})^{+}$ is the identity component of $H(\mathbb{R})$. Define
$$
X_{K}:=H(\mathbb{Q})\backslash\left(\mathbb{D}\times H(\mathbb{A}_{f})/K\right).
$$
This is the set of complex points of a quasi-projective variety rational over $\mathbb{Q}$, and if $\Gamma_{K}=H(\mathbb{Q})\cap H(\mathbb{R})^{+}K$, then $X_{K}\cong \Gamma_{K}\backslash \mathbb{D}^{+}$ via $[z,h]\to[\gamma^{-1}z]$, where $\mathbb{D}^{+}\subset\mathbb{D}$ is the subset of positively oriented negative 2-planes, and $h=\gamma k$ for some $\gamma\in H(\mathbb{Q})^{+}$ and $k\in K$ by the strong approximation theorem.

Assume that 
$$
V(\mathbb{R})=V_{0}+\mathbb{R}e+\mathbb{R}f
$$
where $e$ and $f$ are such that $Q(e)=Q(f)=0$ and $(e,f)=1$. Then the signature of $V_{0}$ is $(n-1,1)$ and for the negative cone 
$$
\mathcal{C}=\{y\in V_{0}|\, Q(y)<0\},
$$
we have
$$
\mathbb{D}\cong \mathcal{H}:=\{z\in V_{0}(\mathbb{C})|\, {\rm Im}(z)\in\mathcal{C}\}.
$$
The isomorphism is given by $z\to w(z):=e-Q(z)f+z$ composed with projection to $\mathcal{L}$. The map $z\to w(z)$ can be viewed as a holomorphic section of $\mathcal{L}$.

\begin{Example}
\label{example}
Let $V=\{X\in {\rm M}_{2}(\mathbb{Q})|\,{\rm tr}(X)=0\}$ with quadratic form $Q(X)=N\det(X)$ of signature $(1,2)$. Then $\mathbb{D}\cong \mathbb{H}^{\pm}$ and $H={\rm GSpin}(V)={\rm GL}_{2}$. Now let 
$$
K=\left\{\left.\begin{pmatrix}a&b\\c&d\end{pmatrix}\in{\rm GL}_{2}(\hat{\mathbb{Z}})\right|\,c\equiv0\pmod{N}\right\}.
$$
Then we have $H(\mathbb{Q})\cap H(\mathbb{R})^{+}K=\Gamma_{0}(N)$ and $X_{K}\cong \Gamma_{0}(N)\backslash\mathbb{H}=Y_{0}(N)$ via $[z,hk]\to [h^{-1}\cdot z]$.

\end{Example}

\begin{Definition}
A modular form on $\mathbb{D}\times H(\mathbb{A}_{f})$ of weight $k$ is a meromorphic function $f:\mathbb{D}\times H(\mathbb{A}_{f})\to\mathbb{C}$ such that
\begin{enumerate}[(1)]
\item{$f(z,hk)=f(z,h)$ for all $k\in K$,}
\item{$f(\gamma{z},\gamma{h})=j(\gamma,z)^{k}f(z,h)$ for all $\gamma\in H(\mathbb{Q})$, where $j(\gamma,z)$ is the automorphy factor induced by the isomorphism $w$.}
\end{enumerate}
\end{Definition}

\subsection{Regularized Theta Lift and Borcherds Form}
For $z\in\mathbb{D}$, let ${\rm pr}_{z}:V(\mathbb{R})\to z$ be the projection map, and for $x\in V(\mathbb{R})$, let $R(x,z)=-({\rm pr}_{z}(x),{\rm pr}_{z}(x))$. Then we define 
$$
(x,x)_{z}=(x,x)+2R(x,z),
$$
and our Gaussian for $V$ is the function
$$
\varphi_{\infty}(x,z)=e^{-\pi(x,x)_{z}}.
$$
For $\tau\in\mathbb{H}$ with $\tau=u+iv$, let
$$
g_{\tau}=\begin{pmatrix}1&u\\0&1\end{pmatrix}\begin{pmatrix}v^{\frac{1}{2}}&0\\0&v^{-\frac{1}{2}}\end{pmatrix},
$$
and $g_{\tau}'=(g_{\tau},1)\in {\rm Mp}_{2}(\mathbb{R})$. Let $l=\frac{n}{2}-1$, $G={\rm SL}_{2}$ and $\rho$ be the Weil representation of the metaplectic group $G_{\mathbb{A}}'$ on $\mathcal{S}(V(\mathbb{A}_{f}))$, the Schwartz space of $V(\mathbb{A}_{f})$. Then for the linear action of $H(\mathbb{A}_{f})$ we write $\rho(h)\varphi(x)=\rho(h^{-1}x)$ for $\varphi\in\mathcal{S}(V(\mathbb{A}_{f}))$. For $z\in\mathbb{D}$ and $h\in H(\mathbb{A}_{f})$, we have the linear functional on $\mathcal{S}(V(\mathbb{A}_{f}))$ given by
$$
\varphi\to\theta(\tau,z,h;\varphi):=v^{-\frac{1}{2}}\sum_{x\in V(\mathbb{Q})}\rho(g_{\tau}')\left(\varphi_{\infty}(\cdot,z)\otimes\rho(h)\varphi\right)(x).
$$

Let $L$ be a lattice of $V$, and let $L'$ be the dual lattice of $L$ defined by
$$
L'=\{x\in V|\,(x,L)\subset\mathbb{Z}\}.
$$
Let $\mathcal{S}_{L}$ be the subspace of $\mathcal{S}(V(\mathbb{A}_{f}))$ consisting of functions with support in $\hat{L}'$ and constant on cosets of $\hat{L}$, where $\hat{L}=L\otimes_{\mathbb{Z}}\hat{\mathbb{Z}}$. Then 
$$
\mathcal{S}_{L}=\bigoplus_{\eta\in L'/L}\mathbb{C}\phi_{\eta},\quad \phi_{\eta}={\rm Char}(\eta+\hat{L}).
$$

Let $\Gamma'={\rm Mp}_{2}(\mathbb{Z})$ be the full inverse image of ${\rm SL}_{2}(\mathbb{Z})\subset G(\mathbb{R})$ in ${\rm Mp}_{2}(\mathbb{R})$. 
\begin{Definition}
A function $\vec{F}:\mathbb{H}\to \mathcal{S}_{L}$ is a weakly holomorphic modular form of weight $1-\frac{n}{2}$ and type $\rho_{L}$ for $\Gamma'$ if
\begin{enumerate}[(i)]
\item{$\vec{F}(\gamma'\tau)=(c\tau+d)^{1-\frac{n}{2}}\rho_{L}(\gamma')\vec{F}(\tau)$ for all $\gamma'\in\Gamma'$,}
\item{$\vec{F}(\tau)$ has a Fourier expansion
$$
\vec{F}(\tau)=\sum_{\eta\in L'/L}\sum_{\substack{m\in Q(\eta)+\mathbb{Z}\\m\gg-\infty}}c(m,\eta)q^{m}\phi_{\eta}
$$
where the condition $m\equiv Q(\eta)\pmod{\mathbb{Z}}$ follows from the transformation law for $T'\to\begin{pmatrix}1&1\\0&1\end{pmatrix}$.
}
\end{enumerate}
\end{Definition}
For the theta function
$$
\theta(\tau,z,h)=\sum_{\mu\in L'/L}\theta(\tau,z,h;\phi_{\mu}),
$$
we can pair it with $\vec{F}(\tau)$ by the following  $\mathbb{C}$-bilinear pairing
$$
\langle \vec{F}(\tau),\theta(\tau,z,h)\rangle=\sum_{\mu\in L'/L}\sum_{m\in Q(\mu)+\mathbb{Z}}c(m,\mu)q^{m}\theta(\tau,z,h;\phi_{\mu}).
$$
Using this pairing, we define a regularized integral as in \cite{BO}, called theta lift,
$$
\Phi(z,h;\vec{F}):=\underset{s=0}{{\rm CT}}\left\{\lim_{t\to\infty}\int_{\mathcal{F}_{t}}\langle \vec{F}(\tau),\theta(\tau,z,h)\rangle v^{-2}dudv\right\}
$$
where $\underset{s=0}{{\rm CT}}$ denotes the constant term in the Laurent expansion at $s=0$ of 
$$
\lim_{t\to\infty}\int_{\mathcal{F}_{t}}\langle \vec{F}(\tau),\theta(\tau,z,h)\rangle v^{-2}dudv,
$$
$\mathcal{F}_{t}$ is the truncated fundamental domain defined by
$$
\mathcal{F}_{t}:=\{\tau\in\mathcal{F}|\, {\rm Im}(\tau)\leq t\}
$$
and $\mathcal{F}$ is the usual fundamental domain for the action of ${\rm SL}_{2}(\mathbb{Z})$ on $\mathbb{H}$.

\begin{Theorem}[Borcherds]
There is a meromorphic modular form $\Psi(z,h;\vec{F})$ of weight $\frac{1}{2}c(0,0)$ on $\mathbb{D}\times H(\mathbb{A}_{f})$ such that
$$
\Phi(z,h;\vec{F})=-2\log|\Psi(z,h;\vec{F})|^{2}|y|^{c(0,0)}-c(0,0)\left(\log(2\pi)+\Gamma'(1)\right)
$$
where $y={\rm Im}(z)$. Such a meromorphic modular form is called a Borcherds form arising from the regularized theta lift of a modular form $\vec{F}$.
\end{Theorem}

The following results due to Kim \cite{KI} give $\Gamma_{0}^{*}(p)$ extensions of \eqref{gzbor} for $j_{p}^{*}(\tau)$.
\begin{Theorem}[Kim]
\label{kim}
Let $p$ be the prime such that $\Gamma^{*}_{0}(p)$ is of genus zero, and let $j_{p}^{*}(\tau)$ be a Hauptmodul for $\Gamma^{*}_{0}(p)$. Let $V=\{X\in {\rm M}_{2}(\mathbb{Q})|\,{\rm tr}(X)=0\}$ with quadratic form $Q(X)=p\det(X)$ and $L=\begin{pmatrix}\mathbb{Z}&\mathbb{Z}\\p\mathbb{Z}&\mathbb{Z}\end{pmatrix}\cap V$. Then there is a weakly holomorphic modular form $\vec{F}_{d,p,\beta}(\tau)$ for $\rho_{L}$ with $c(0,\eta)=0$ and principal part $q^{-d}\phi_{\beta}+q^{-d}\phi_{-\beta}$ such that
$$
\prod_{Q_{d}\in\mathcal{Q}^{+}_{d,p,\beta}/\Gamma_{0}(p)}\left(j_{p}^{*}(\tau)-j_{p}^{*}(\tau_{Q_{d}})\right)^{\frac{2}{|\bar{\Gamma}_{0}(p)_{Q_{d}}|}}=\Psi(z,h;\vec{F}_{d,p,\beta})
$$
under the identification given in Example \ref{example}, where $\mathcal{Q}_{d,p,\beta}^{+}$ denote the set of positive definite quadratic forms $aX^{2}+bXY+cY^{2}$ of negative discriminant $-d$ with $p|a$, $b\equiv\beta\pmod{2p}$ and $-d=b^{2}-4ac$, and $\bar{\Gamma}_{0}(p)_{Q_{d}}$ is the stabilizer of $Q_{d}$ in $\Gamma_{0}(p)/\{\pm I\}$.
When $-d$ is a fundamental discriminant and $-d<-4$, we have
$$
\prod_{Q_{d}\in\mathcal{Q}^{prim,+}_{d,p,\beta}/\Gamma_{0}(p)}\left(j_{p}^{*}(\tau)-j_{p}^{*}(\tau_{Q_{d}})\right)^{2}=\Psi(z,h;\vec{F}_{d,p,\beta}).
$$
\end{Theorem}

\begin{Remark}
\label{rem2}
  Similar relation may not exist for a Hauptmodul $j_{p}(\tau)$ on $\Gamma_{0}(p)$. It is mainly due to that under the $\rm{O}(1,2)$ setup, the divisor of a Borcherds form must be a linear combination of the special divisors of the form
$$
\sum_{Q_{d}\in\mathcal{Q}_{d,p,\beta}^{+}/\Gamma_{0}(p)}\frac{1}{|\bar{\Gamma}_{0}(p)_{Q_{d}}|}[\tau_{Q_{d}}],
$$
and this partially determines some properties of the theta lift input $\vec{F}$, which somehow determines the behavior of the associated Borcherds form. However, by \cite[Cor. 5.4]{BS}, one can check that such a Borcherds form must have the same behavior at the cusps $i\infty$ and $0$ under the identification between $X_{K}$ and $Y_{0}(p)$. This could not happen to the function of the form
$$
\prod_{Q\in\mathcal{Q}/\Gamma_{0}(p)}\left(j_{p}(\tau)-j_{p}(\tau_{Q})\right)^{\frac{1}{|\bar{\Gamma}_{0}(p)_{Q}|}}
$$
for any set of binary quadratic forms $\mathcal{Q}$, since $j_{p}(\tau)$ has different behavior at the cusps $i\infty$ and $0$. Since similar Borcherds form ``relation'' may not exist for $j_{p}(\tau)$, the small CM value formula will not work.
\end{Remark}

\section{Small CM Value Formula}
\label{schofer}
In this section, we briefly review Schofer's small CM value formula \cite{S} and certain related key concepts, such as CM-cycle and Fourier coefficients of an incoherent Eisenstein series of weight 1.

\subsection{CM-Cycle} Assume that we have a rational splitting $V=V_{+}\oplus U$, where $V_{+}$ is of signature $(n,0)$ and $U$ is of signature $(0,2)$. Then $U$ gives rise to a two-point subset $\mathbb{D}_{0}$ of $\mathbb{D}$. Let $T={\rm GSpin}(U)$ and let $K_{T}=K\cap T(\mathbb{A}_{f})$.. Then there is an embedding $T\hookrightarrow H$ and we have a CM-cycle of $X_{K}$,
$$
Z(U)_{K}:=T(\mathbb{Q})\backslash\left(\mathbb{D}_{0}\times T(\mathbb{A}_{f})/K_{T}\right)\hookrightarrow X_{K},
$$
which is a $0$-cycle.
\subsection{Eisenstein Series and Small CM Value Formula} Assume that $V$ is of signature $(n,2)$ with $n$ even. Inside of $G_{\mathbb{A}}$, we have the subgroups
$$
N_{\mathbb{A}}:=\{n(b)|\,b\in\mathbb{A}\},\quad n(b)=\begin{pmatrix}1&b\\0&1\end{pmatrix},
$$
and 
$$
M_{\mathbb{A}}:=\{m(a)|\,a\in\mathbb{A}^{\times}\},\quad m(a)=\begin{pmatrix}a&0\\0&a^{-1}\end{pmatrix}.
$$
Define the quadratic character $\chi=\chi_{V}$ of $\mathbb{A}^{\times}/\mathbb{Q}^{\times}$ via the global Hilbert symbol by
$$
\chi(x)=(x,-\det(V)),
$$
where $\det(V)\in\mathbb{Q}^{\times}/(\mathbb{Q}^{\times})^{2}$ is the determinant of the matrix for the quadratic form $Q$ on $V$. For $s\in\mathbb{C}$, let $I(s,\chi)$ be the principal series representation of $G_{\mathbb{A}}$. This space consists of smooth functions $\Phi(s)$ on $G_{\mathbb{A}}$ such that 
$$
\Phi(n(b)m(a)g,s)=\chi(a)|a|^{s+1}\Phi(g,s).
$$
We have a $G_{\mathbb{A}}$-intertwining map
$$
\lambda=\lambda_{V}:\mathcal{S}(V(\mathbb{A}_{f}))\to I\left(\frac{n}{2},\chi\right),
$$
where $\lambda(\varphi)(g)=(\rho(g)\varphi)(0)$. If $K_{\infty}=SO(2)$ and $K_{f}={\rm SL}_{2}(\hat{\mathbb{Z}})$, then a section section $\Phi(s)\in I(s,\chi)$ is called standard if its restriction to $K_{\infty}K_{f}$ is independent of $s$. The function $\lambda(\varphi)$ has a unique extension to a standard section $\Phi(s)\in I(s,\chi)$ such that $\Phi(\frac{n}{2})=\lambda(\varphi)$. We let $P=MN$ and define the Eisenstein series associated to $\Phi(s)$ by
$$
E(g,z;\Phi)=\sum_{\gamma\in P(\mathbb{Q})\backslash G(\mathbb{Q})}\Phi(\gamma g,s).
$$
This series converges for $\Re(s)>1$ and has a meromorphic continuation to the whole $s$-plane.

For $r\in\mathbb{Z}$, let $\chi_{r}$ be the character of $K_{\infty}$ defined by
$$
\chi_{r}(k_{\theta})=e^{ir\theta},\quad k_{\theta}=\begin{pmatrix}\cos\theta&\sin\theta\\-\sin\theta&\cos\theta\end{pmatrix}\in K_{\infty}.
$$
Let $\phi: G(\mathbb{R})\to\mathbb{C}$ be a smooth function of weight $l$, meaning $\phi(gk_{\theta})=\chi_{l}(k_\theta)\phi(g)$. Let $\Phi^{l}_{\infty}(s)$ be the normalized eigenfunction of weight $l$ for $K_{\infty}$, i.e.,
$$
\Phi^{l}_{\infty}(gk,s)=\chi_{l}(k)\Phi(g,s).
$$
Now take $\Phi(s)=\Phi^{l}_{\infty}(s)\otimes \lambda(\varphi)$. By strong approximation, the series $E(g,s;\Phi)$ is determined by the Eisenstein series
$$
E(\tau,s;\varphi,l):=v^{-\frac{1}{2}}E(g_{\tau},s;\Phi^{l}_{\infty}\otimes\lambda(\varphi))
$$
which is a non-holomorphic modular form of weight $l$ on $\mathbb{H}$.

\begin{Definition}
\label{defE}
Consider $V=U$ of signature $(0,2)$ and view $U\cong\mathbb{Q}(\sqrt{-D})$ with discriminant $-D$. Let $\chi_{D}$ be the character of $\mathbb{Q}_{\mathbb{A}}^{\times}$ defined via the global Hilbert symbol by $\chi_{D}(x)=(x,-D)_{\mathbb{A}}$. For $\varphi\in\mathcal{S}(U(\mathbb{A}_{f}))$, let 
$$
E(\tau,s;\varphi,1)=\sum_{m\in\mathbb{Q}}A_{\varphi}(s,m,v)q^{m},
$$
where the Fourier coefficients have Laurent expansions
$$
A_{\varphi}(s,m,v)=b_{\varphi}(m,v)s+O(s^{2})
$$
at $s=0$. For $m\geq0$, define
$$
\kappa_{\varphi}(m)=A_{\varphi}'(0,m,v).
$$
\end{Definition}

Let $k=\mathbb{Q}(\sqrt{-D})$ be an imaginary quadratic field with fundamental discriminant $-D$, ring of integers $\mathcal{O}_{k}$ and different $\partial$. Let $\chi$ be the quadratic Dirichlet character associated to $k$. Let $\mathfrak{a}$ be a fractional ideal of $k$ and let $L=\mathfrak{a}$ be a lattice with integral quadratic form $Q(x)=-{\rm N}(x)/{\rm N}(\mathfrak{a})$. Then $L'=\partial^{-1}\mathfrak{a}$. The following theorem is well known (see, e.g., \cite{KY}).

\begin{Theorem}
\label{coeff}
Assume that $m>0$ and $m\in Q(\eta)+\mathbb{Z}$ for some $\eta\in L'/L$. Define $o(m)$ to be the number of primes $p|D$ such that ${\rm ord}_{p}(mD)>0$, and define
$$
\rho(m)=|\{\mbox{ideal $\mathfrak{b}\subset\mathcal{O}_{k}$}|\, {\rm N}(\mathfrak{b})=m\}|.
$$
 Let ${\rm Diff}(m)$ denote the set of primes $p<\infty$ such that $\chi_{p}(-m{\rm N}(\mathfrak{a}))=-1$. Then $\kappa_{\mu}(m)=0$ unless $|{\rm Diff}(m)|=1$. Assume that ${\rm Diff}(m)=\{p\}$. Then
\begin{enumerate}[(i)]
\item{if $p$ is inert in $k$, then
$$
\kappa_{\eta}(m)=-\frac{2^{o(m)-1}w(k)}{h(-D)}\left({\rm ord}_{p}(m)+1\right)\rho(mD/p)\log{p},
$$
}

\item{if $p$ is ramified in $k$, i.e., $p|D$, then
$$
\kappa_{\eta}(m)=-\frac{2^{o(m)-1}w(k)}{h(-D)}{\rm ord}_{p}(mD)\rho(mD)\log{p}.
$$
}
For $m=0$, one has
$$
\kappa_{\eta}(0)=\delta_{0,\eta}\left(\log{v}-2\frac{\Lambda'(0,\chi)}{\Lambda(0,\chi)}\right),
$$
where $\Lambda(s,\chi)$ is the complete $L$-function associated to $\chi$.
\end{enumerate}
\end{Theorem}

We end this section with Schofer's small CM value formula \cite[Theorem 3.1]{S}.
\begin{Theorem}[Schofer]
\label{scho}
Let $\vec{F}:\mathbb{H}\to\mathcal{S}_{L}$ be a weakly holomorphic modular form for $\rho_{L}$ of weight $1-\frac{n}{2}$ with Fourier expansion
$$
\vec{F}(\tau)=\sum_{\eta\in L'/L}\sum_{m\in Q(\eta)+\mathbb{Z}}c(m,\eta)q^{m}\phi_{\eta}
$$
and $c(0,0)=0$. For $V=V_{+}\oplus U$, where $V_{+}=U^{\perp}$ of signature $(n,0)$, write $L_{\pm}$ for $L\cap V_{+}$ and $L\cap U$, respectively. Let ${\rm pr}_{\pm}$ denote the projections of $V$ onto $V_{+}$ and $U$, respectively, and write $x_{\pm}$ for ${\rm pr}_{\pm}(x)$ for $x\in V$. Then
$$
\sum_{z\in Z(U)_{K}}\log|\Psi(z;\vec{F})|^{2}=-{|T(\mathbb{Q})\backslash T(\mathbb{A}_{f})/K_{T}|}\sum_{\eta\in L'/L}\sum_{m\geq0}c(-m,\eta)K_{\eta}(m)
$$
where
\begin{equation}
\label{KM}
K_{\eta}(m)=\sum_{\lambda\in L/(L_{+}+L_{-})}\sum_{\ell\in\eta_{+}+\lambda_{+}+L_{+}}\kappa_{\eta_{-}+\lambda_{-}}(m-Q(\ell)).
\end{equation}
\end{Theorem}

\section{Proof of Theorem \ref{main}}

Throughout this section, let $V=\{x\in {\rm M}_{2}(\mathbb{Q})|\, {\rm tr}(x)=0\}$ with quadratic form $Q(x)=p\det(x)$, and let 
$$
L=\left\{\begin{pmatrix}b&-\frac{a}{p}\\c&-b\end{pmatrix}|\,a,b,c\in\mathbb{Z}\right\}.
$$
The dual lattice is
$$
L'=\left\{\begin{pmatrix}\frac{b}{2p}&-\frac{a}{p}\\c&-\frac{b}{2p}\end{pmatrix}|\,a,b,c\in\mathbb{Z}\right\},
$$
and so
$$
L'/L=\left\{\begin{pmatrix}\frac{\eta}{2p}&0\\0&-\frac{\eta}{2p}\end{pmatrix}|\,\eta\in\mathbb{Z}/2p\mathbb{Z}\right\}.
$$
We write $\eta$, when there is no confusion, for  $\begin{pmatrix}\frac{\eta}{2p}&0\\0&-\frac{\eta}{2p}\end{pmatrix}$. 

\subsection{Interpretation of $Z(U_{x})_{K}$} As we have seen in Example \ref{example}, we have $H={\rm GSpin}(V)={\rm GL}_{2}$ and 
$$
X_{K}=H(\mathbb{Q})\backslash\left(\mathbb{D}\times H(\mathbb{A}_{f})/K\right)\cong \Gamma_{0}(p)\backslash\mathbb{H}=Y_{0}(p).
$$
For $r\in\mathbb{Q}_{>0}$ and $\eta\in\mathbb{Z}/2p\mathbb{Z}$, define
$$
L_{\eta}(r)=\{x\in\eta+L|\,Q(x)=r\}.
$$
Now let $-D$ be a negative fundamental discriminant with $-D=\mu^{2}+{4pm}$ for some $\mu\in\mathbb{Z}/2p\mathbb{Z}$. Take $x\in L_{\mu}(D/4p)$ to be
$$
x=\begin{pmatrix}\frac{\mu}{2p}&\frac{1}{p}\\m&-\frac{\mu}{2p}\end{pmatrix}.
$$
 We can obtain an associated subspace $U_{x}=x^{\perp}$ of signature $(0,2)$, and the associated CM-cycle of $X_{K}$
$$
Z(U_{x})_{K}=T(\mathbb{Q})\backslash\left(\mathbb{D}_{x}\times T(\mathbb{A}_{f})/K_{T}\right)\hookrightarrow H(\mathbb{Q})\backslash\left(\mathbb{D}\times H(\mathbb{A}_{f})/K\right)\cong Y_{0}(p).
$$
where $T={\rm GSpin}(U_{x})$, $\mathbb{D}_{x}$ is the two-point subset of $\mathbb{D}$ given by $U_{x}$ and $K_{T}=K\cap T(\mathbb{A}_{f})$. In particular, $T(\mathbb{Q})\cong \mathbb{Q}(\sqrt{-D})^{\times}$. In this subsection, we aim to interpret $Z(U_{x})_{K}$ as a CM-cycle of $Y_{0}(p)$. We reply heavily on \cite[Section 6.3]{S2}.

First, we note that $x$ gives rise to a CM point $z_{x}=\frac{\mu+\sqrt{-D}}{2pm}$ in $\mathbb{H}$.
Let $k$ denote the imaginary quadratic field $\mathbb{Q}(\sqrt{-D})$. By the theory of complex multiplication \cite{Sh}, there is an embedding $\phi_{x}:k\hookrightarrow {\rm M}_{2}(\mathbb{Q})$ such that
$$
\phi_{x}(u)\begin{pmatrix}\frac{\mu+\sqrt{-D}}{2pm}\\1\end{pmatrix}=u\begin{pmatrix}\frac{\mu+\sqrt{-D}}{2pm}\\1\end{pmatrix}.
$$
Then one can check that
$$
\phi_{x}\left(\frac{-D+\sqrt{-D}}{2}\right)=\begin{pmatrix}\frac{\mu-D}{2}&\frac{-D-\mu^{2}}{4pm}\\pm&\frac{-D-\mu}{2}\end{pmatrix}.
$$
Thus,
\begin{equation}
\label{phix}
\phi_{x}\left(v+w\frac{-D+\sqrt{-D}}{2}\right)=\begin{pmatrix}v+\frac{\mu-D}{2}w&\frac{-D-\mu^{2}}{4pm}w\\pmw&v+\frac{-D-\mu}{2}w\end{pmatrix}.
\end{equation}
Let $R_{0}=\begin{pmatrix}\mathbb{Z}&\mathbb{Z}\\p\mathbb{Z}&\mathbb{Z}\end{pmatrix}$ be an order of conductor $p$ of ${\rm M}_{2}(\mathbb{Q})$. Then one can check that $\phi_{x}^{-1}(R_{0})=\mathcal{O}_{k}$ since $D$ and $\mu$ have the same parity, and $\mathcal{O}_{k}$ is the maximal order of $k$.

\begin{Lemma}
\label{KO}
$K_{T}\cong\hat{\mathcal{O}}_{k}^{\times}$.
\end{Lemma}

\begin{proof}
By \eqref{phix}, one can check that $\phi_{x}^{-1}(K)=\hat{\mathcal{O}}_{k}$. So $\hat{\mathcal{O}}_{k}^{\times}\subset \phi_{x}^{-1}(K_{T})$. Since ${\mathcal{O}}_{k}$ is the maximal order of $k$, then $\hat{\mathcal{O}}_{k}^{\times}=\phi_{x}^{-1}(K_{T})$.
\end{proof}

\begin{Proposition}
\label{ZU}
The CM-cycle $Z(U_{x})_{K}$ of $X_{K}$
$$
Z(U_{x})_{K}=T(\mathbb{Q})\backslash\left(\mathbb{D}_{x}\times T(\mathbb{A}_{f})/K_{T}\right)
$$
is identified with
$$
Z(D/4p,\mu):=\sum_{Q_{D}\in \mathcal{Q}^{prim}_{D,p,\mu}/\Gamma_{0}(p)}[\tau_{Q_{D}}]=2\sum_{Q_{D}\in \mathcal{Q}^{prim,+}_{D,p,\mu}/\Gamma_{0}(p)}[\tau_{Q_{D}}]
$$
in $Y_{0}(p)$. 
\end{Proposition}

\begin{proof}
First it is known by Witt's Theorem that for any $\tilde{x}\in L_{\mu}(D/4p)$, there is a $\gamma\in H(\mathbb{Q})$ such that $\tilde{x}=\gamma\cdot x$. Let 
$$
M=\{R\subset {\rm M}_{2}(\mathbb{Q})|\,\mbox{$M$ is an order and $\phi_{x}^{-1}(R)=\mathcal{O}_{k}$}\}.
$$
Then by \cite[Prop. 6.18]{S2}, there is a well-defined surjective map
$$
L_{\mu}(D/4p)\sqcup L_{-\mu}(D/4p)\to T(\mathbb{Q})\backslash M
$$
sending $\tilde{x}\to[\gamma^{-1}R_{0}]$. So we have a well-defined surjective map
$$
\Gamma_{p}\backslash \left(L_{\mu}(D/4p)\sqcup L_{-\mu}(D/4p)\right)\to T(\mathbb{Q})\backslash M,
$$
where 
$$
\Gamma_{p}=\left\{\left.\begin{pmatrix}a&b\\c&d\end{pmatrix}\in {\rm GL}_{2}(\mathbb{Z})\right|\, c\equiv0\pmod{p}\right\}.
$$
  Since $T(\mathbb{A}_{f})$ acts transitively on $M$ \cite[Thm. 6.16]{S2}, and the stabilizer of $M$ in $H(\mathbb{A}_{f})$ is $KZ(\mathbb{A}_{f})$, then one has 
$$
M\cong T(\mathbb{A}_{f})/K_{T}Z(\mathbb{A}_{f}).
$$
Since $H={\rm GL}_{2}$, then $Z(\mathbb{A}_{f})=\mathbb{Q}^{\times}\hat{\mathbb{Z}}^{\times}$.
Furthermore, by Lemma \ref{KO}, we have $K_{T}\cong\hat{\mathcal{O}}_{k}^{\times}$. So all of these tell that
$$
M\cong T(\mathbb{A}_{f})/\hat{\mathcal{O}}_{k}^{\times}\cong T(\mathbb{A}_{f})/K_{T},
$$
and thus we have a well-defined surjective map
$$
\Gamma_{p}\backslash \left(L_{\mu}(D/4p)\sqcup L_{-\mu}(D/4p)\right)\to T(\mathbb{Q})\backslash T(\mathbb{A}_{f})/K_{T}\cong k^{\times}\backslash k_{\mathbb{A}_{f}}^{\times}/\hat{\mathcal{O}}_{k}^{\times}.
$$
In addition, one can check that 
$$
\Gamma_{p}\backslash\left( L_{\mu}(D/4p)\sqcup L_{-\mu}(D/4p)\right)\cong\Gamma_{0}(p)\backslash \mathcal{Q}_{D,p,\mu}^{prim,+},
$$
and by \cite[Lem. 2]{HZ} or \cite[Prop. p. 505]{GKZ}, it is known that $|\Gamma_{0}(p)\backslash \mathcal{Q}_{D,p,\mu}^{prim,+}|=h(-D)$, where $h(-D)$ denotes the class number of the imaginary quadratic field $\mathbb{Q}(\sqrt{-D})$. So we have the isomorphisms
$$
\Gamma_{p}\backslash\left( L_{\mu}(D/4p)\sqcup L_{-\mu}(D/4p)\right)\cong T(\mathbb{Q})\backslash T(\mathbb{A}_{f})/K_{T}\cong k^{\times}\backslash k_{\mathbb{A}_{f}}^{\times}/\hat{\mathcal{O}}_{k}^{\times}.
$$
Finally, this together with \cite[Prop. 6.23]{S2} implies Proposition \ref{ZU}.
\end{proof}

\begin{Proposition}
\label{proplog}
Let $p$ be a prime such that $\Gamma^{*}_{0}(p)$ is of genus zero, and let $j_{p}^{*}(\tau)$ be a Hauptmodul for $\Gamma^{*}_{0}(p)$. 
\begin{align*}
&\log\left(\prod_{Q_{D}\in \mathcal{Q}^{prim,+}_{D,p,\mu}/\Gamma_{0}(p)}\prod_{Q_{d}\in \mathcal{Q}^{prim,+}_{d,p,\beta}/\Gamma_{0}(p)}|j_{p}^{*}(\tau_{Q_{D}})-j_{p}^{*}(\tau_{Q_{d}})|^{8}\right)\\
&=-{h(-D)}\left(K_{\beta}(d/4p)+K_{-\beta}(d/4p)\right),
\end{align*}
where $h(-D)$ denotes the class number of the imaginary quadratic field $\mathbb{Q}(\sqrt{-D})$, and $K_{\eta}(m)$ is defined as in \eqref{KM}.
\end{Proposition}

\begin{proof}
By Theorems \ref{kim} and \ref{scho}, we can deduce that
\begin{align*}
&\sum_{z\in Z(U_{x})_{K}}\left(\log \left|\Psi\left(z;\vec{F}_{d,p,\beta}\right)\right|^{2}\right)\\
&=-{|T(\mathbb{Q})\backslash T(\mathbb{A}_{f})/K_{T}|}\left(K_{\beta}(d/4p)+K_{-\beta}(d/4p)\right).
\end{align*}
This together with Lemma \ref{KO} and Proposition \ref{ZU} implies the desired formula
\begin{align*}
&\sum_{Q_{D}\in \mathcal{Q}^{prim,+}_{D,p,\mu}/\Gamma_{0}(p)}\left(\log \prod_{Q_{d}\in \mathcal{Q}^{prim,+}_{d,p,\beta}/\Gamma_{0}(p)}|j_{p}^{*}(\tau_{Q_{D}})-j_{p}^{*}(\tau_{Q_{d}})|^{8}\right)\\
&=-{h(-D)}\left(K_{\beta}(d/4p)+K_{-\beta}(d/4p)\right).
\end{align*}
\end{proof}
\begin{proof}[Proof of Theorem \ref{main}]
These follow from Proposition \ref{proplog}, Theorem \ref{coeff}, \eqref{KM}, and \eqref{K1} as computed in the following subsection.
\end{proof}

\subsection{Lattice Computations}  In this subsection, we explicitly compute the lattices $L_{+}$ and $L_{-}$ for our case. Recall that
$$
  x=\begin{pmatrix}\frac{\mu}{2p}&\frac{1}{p}\\m&-\frac{\mu}{2p}\end{pmatrix}.
  $$
    We can easily compute that $U_{x}=\mathbb{Q}e_{1}+\mathbb{Q}e_{2}$ where
$$
e_{1}=\begin{pmatrix}1&0\\-\mu&-1\end{pmatrix}\quad\mbox{and}\quad e_{2}=\begin{pmatrix}0&\frac{1}{p}\\-m&0\end{pmatrix}.
$$
Then one can easily check that
\begin{equation}
\label{L+}
L_{+}=L\cap \mathbb{Q}x=\mathbb{Z}\frac{2p}{g}x,
\end{equation}
where $g=\gcd(\mu,2p)$,
and
\begin{equation}
\label{L-}
L_{-}=L\cap U_{x}=\mathbb{Z}e_{1}+\mathbb{Z}e_{2}.
\end{equation}
Then $L_{-}$ can be identified with the ideal $\mathfrak{a}=[p,\frac{\mu+\sqrt{-D}}{2}]$ in $\mathcal{O}_{k}$ as a quadratic $\mathbb{Z}$-lattice with the quadratic form $Q(z)~=~-\frac{z\bar{z}}{p}$ as mentioned after Definition~\ref{defE}.
Recall that ${\rm pr}_{\pm}$ denotes the projections of $V$ onto $\mathbb{Q}x$ and $U_{x}$, respectively, and $\lambda_{\pm}={\rm pr}_{\pm}(\lambda)$. 
For $\lambda=\begin{pmatrix}b&-a/p\\c&-b\end{pmatrix}\in L$, we have
$$
\lambda_{+}=\frac{2p(am-b\mu-c)}{D}x
$$
and
$$
\lambda_{-}=\frac{Db-\mu(am-b\mu-c)}{D}e_{1}+\frac{-Da-2p(am-b\mu-c)}{D}e_{2}.
$$
Setting $y=am-b\mu-c$ and modding out by $L_{+}+L_{-}$, we obtain, by \eqref{L+} and \eqref{L-},
\begin{equation}
\label{LLL}
L/(L_{+}+L_{-})=\{\lambda+(L_{+}+L_{-})\}=\left\{\frac{2py}{D}x+\left(-\frac{\mu y}{D}\right)e_{1}+\left(-\frac{2py}{D}\right)e_{2}\right\},
\end{equation}
where $0\leq y<\frac{D}{g}$.
Also, for $\beta\in L'/L$ with
$$
\beta=\begin{pmatrix}\beta/2p&0\\0&-\beta/2p\end{pmatrix},
$$
we have 
$$
\beta_{+}=-\frac{\mu\beta}{D}x\quad\mbox{and}\quad \beta_{-}=-\frac{2m\beta}{D}e_{1}+\frac{\mu\beta}{D}e_{2}.
$$
Then we can express $\beta_{+}+\lambda_{+}$ as
\begin{align*}
\beta_{+}+\lambda_{+}&=\frac{2py-\mu\beta}{D}\begin{pmatrix}\frac{\mu}{2p}&\frac{1}{p}\\m&-\frac{\mu}{2p}\end{pmatrix},
\end{align*}
and the elements of $\beta_{+}+\lambda_{+}+L_{+}$ can be written as
$$
\ell=\frac{2py-\mu\beta}{D}\begin{pmatrix}\frac{\mu}{2p}&\frac{1}{p}\\m&-\frac{\mu}{2p}\end{pmatrix}+\frac{2pn}{g}x
$$
for $n\in\mathbb{Z}$.
Thus for $\ell\in\beta_{+}+\lambda_{+}+L_{+}$, we have
\begin{equation}
\label{QL}
Q(\ell)=\frac{(g\mu\beta-2npD-2gpy)^{2}}{4g^{2}pD}.
\end{equation}
Therefore, by \eqref{LLL} and \eqref{QL}, we can rewrite
\begin{align*}
K_{\beta}(d/4p)&=\sum_{\lambda\in L/(L_{+}+L_{-})}\sum_{\ell\in\beta_{+}+\lambda_{+}+L_{+}}\kappa_{\beta_{-}+\lambda_{-}}\left(\frac{d}{4p}-Q(\ell)\right)
\end{align*}
as 
\begin{equation}
\label{K1}
K_{\beta}(d/4p)=\sum_{\substack{n\in\mathbb{Z}\\0\leq y<D/g\\|g\mu\beta-2npD-2gpy|<g\sqrt{dD}}}\kappa_{\beta_{-}+\lambda_{-}}\left(\frac{d}{4p}-\frac{(g\mu\beta-2npD-2gpy)^{2}}{4g^{2}pD}\right).
\end{equation}

\begin{Remark}
If we assume $D$ and $d$ to be coprime, similar to \cite[Section 7.3]{S2}, one may simplify the formula in Theorem \ref{main}. We leave the details to the reader.
\end{Remark}

\section{Hilbert Class Polynomials}
In this section, we close this work with illustrating how to employ Theorem \ref{main} to construct some  Hilbert class polynomials.

Take a negative fundamental discriminant $-\mathcal{D}\equiv\alpha^{2}\pmod{4p}$ for some $\alpha\in\mathbb{Z}/2p\mathbb{Z}$ with $h(-\mathcal{D})=~1$. Assume that for a Hauptmodul $\tilde{j}_{p}^{*}(\tau)$, we have $\tilde{j}_{p}^{*}(i\infty)=\infty$ and $\tilde{j}_{p}^{*}(\tau_{Q_{\mathcal{D}}})=z_{0}$. Then take a linear fractional transformation $\mathcal{T}(\tau)$ such that $\mathcal{T}(\infty)=\infty$ and $\mathcal{T}(z_{0})=0$, and so clearly, $j_{p}^{*}(\tau)=\mathcal{T}(\tilde{j}_{p}^{*}(\tau))$ is a Hauptmodul such that $j_{p}^{*}(\tau_{Q_{\mathcal{D}}})=0$.
Then by Theorem \ref{kim}, one can see that $j_{p}^{*}(\tau)$
is a Borcherds form $\Psi(z,h;\vec{F}_{\mathcal{D},p,\alpha})$.

  For any negative fundamental discriminant $-d\equiv\beta^{2}\pmod{4p}$ with $\beta\in\mathbb{Z}/2p\mathbb{Z}$, we have a Hilbert class polynomial
$$
Y=\prod_{Q_{d}\in\mathcal{Q}_{d,p,\beta}^{prim,+}/\Gamma_{0}(p)}\left(X-j_{p}^{*}(\tau_{Q_{d}})\right)
$$
of degree $h(-d)$. So in order to determine the Hilbert class polynomial explicitly, we need $h(-d)+1$ pairs of values $(X,Y)$. Let $-D$ be a negative fundamental discriminant. If
\begin{equation}
\label{limit}
h(-d)+1\leq |S(p)|,
\end{equation}
where
$$
S(p):=\left\{\text{$-D$}|\,\text{$h(-D)=1$ and $-D\equiv\mu^{2}\pmod{4p}$ for some $\mu\in\mathbb{Z}/2p\mathbb{Z}$}\right\},
$$
then one can use Theorem \ref{main} to get $h(-d)+1$ pairs
$$
X_{D}=\pm|j^{*}_{p}(\tau_{Q_{D}})|\quad\mbox{and}\quad Y_{D}=\pm\prod_{Q_{d}\in\mathcal{Q}_{d,p,\beta}^{prim,+}/\Gamma_{0}(p)}\left|j_{p}^{*}(\tau_{Q_{D}})-j_{p}^{*}(\tau_{Q_{d}})\right|,
$$
where the signs can be determined algebraically.

\begin{Example}
Take $-\mathcal{D}=-11$ with $-11\equiv41^{2}\pmod{188}$. Let $j_{47}^{*}(\tau)$ be a Hauptmodul on $\Gamma_{0}^{*}(47)$ such that $j_{47}^{*}(\tau_{Q_{11}})=0$, where
$$
\tau_{Q_{11}}=\frac{-41+\sqrt{-11}}{94}.
$$
One can check that $S(47)=\{-11,-19,-43,-67,-163\}$,
so we aim to determine a Hilbert class polynomial for $\mathbb{Q}(\sqrt{-39})$ of degree $h(-39)=4$. Following the above procedure and using Theorem~\ref{main}, we can compute that
\begin{align*}
(X_{11},Y_{11})&=(0,1),\quad (X_{19},Y_{19})=(1,1),\quad(X_{43},Y_{43})=(-1,7),\\
&(X_{67},Y_{67})=(2,13),\quad (X_{163},Y_{163})=(4,217).
\end{align*}
Then all of these yield a Hilbert class polynomial for $\mathbb{Q}(\sqrt{-39})$, namely, $Y=X^{4}-X^{3}+2X^{2}-2X+1$.
\end{Example}

\begin{Remark}
Similarly, we can apply Gross-Zagier CM value formula to construct certain Hilbert class polynomials related to $j(\tau)$, and this has been recently pointed out in \cite{E} by Errthum, in which the Hilbert class polynomial of $\mathbb{Q}(\sqrt{-39})$ he obtained is read as
\begin{align*}
Y=&X^{4}+331531596X^{3}-429878960946X^{2}+109873509788637459X\\ 
&+20919104368024767633.
\end{align*}
Apparently, the Hilbert class polynomials induced by $j_{p}^{*}(\tau)$ have some advantages over those for $j(\tau)$ in the sense that they have much smaller coefficients, discriminants and resultants.
\end{Remark}

\begin{Remark}
Another possible way to construct a Hilbert class polynomial via $j_{p}^{*}(\tau)$ is given in \cite{Eh} by Ehlen, in which he gives a formula expressing the value $j_{p}^{*}(\tau_{Q_{d}})$ with $\tau_{Q_{d}}=\frac{\rho+\sqrt{-d}}{2p}$ in terms of the prime ideals of the corresponding Hilbert class field. One can obtain similar formulas for the Galois conjugates of $j_{p}^{*}(\tau_{Q_{d}})$ using Shimura's reciprocity law \cite{Sh}. Understanding the corresponding Hilbert class field and its spectrum, one can explicitly compute $j_{p}^{*}(\tau_{Q_{d}})$ and its Galois conjugates and thus compute the associated Hilbert class polynomial.
\end{Remark}

{\bf Acknowledgment} The author thanks his advisor, Prof. Tonghai Yang, for his support and encouragement. The author would also like to the anonymous reader for his/her helpful comments.

\end{document}